\documentclass[11pt]{article}
\usepackage{latexsym,amsfonts,amsthm,amsmath,amscd,amssymb,color,url}
\usepackage[dvips]{graphicx}

\newtheorem{lemma}{Lemma}[section]
\newtheorem{thm}[lemma]{Theorem}
\newtheorem{rem}[lemma]{Remark}
\newtheorem{prop}[lemma]{Proposition}

\newcommand\matZ{{\mathbb{Z}}}
\newcommand\matC{{\mathbb{C}}}

\newcommand\Sigmatil{{\widetilde\Sigma}}
\newcommand\gtil{{\widetilde{g}}}

\renewcommand{\hbar}{{\overline{h}}}

\newfont{\Got}{eufm10 scaled 1200}
\newcommand{\permu}{{\hbox{\Got S}}}

\newcommand{\compo}{\,{\scriptstyle\circ}\,}

\newcommand{\mycap} [1] {\caption{\footnotesize{#1}}}

\newcommand{\primo}{\textrm{I}}
\newcommand{\secon}{\textrm{I\!I}}
\newcommand{\terzo}{\textrm{I\!I\!I}}
\newcommand{\quart}{\textrm{I\!V}}
\newcommand{\quint}{\textrm{V}}
\newcommand{\sesto}{\textrm{V\!I}}
\newcommand{\setti}{\textrm{V\!I\!I}}

\newcommand{\sistema}[1]{\left\{\begin{array}{l} #1 \end{array}\right.}

\begin{document}

\title{Explicit computation of some\\ families of Hurwitz numbers}

\author{Carlo~\textsc{Petronio}\thanks{Partially supported by INdAM through GNSAGA, by
MIUR through PRIN ``Real and Complex Manifolds: Geometry, Topology and Harmonic Analysis''
and by UniPI through PRA ``Geometria e Topologia delle variet\`a''}}

\maketitle

\begin{abstract}\noindent
We compute the number of (weak) equivalence classes of branched covers from a surface
of genus $g$ to the sphere, with
$3$ branching points, degree $2k$, and local degrees over the branching points
of the form $(2,\ldots,2)$, $(2h+1,1,2,\ldots,2)$, $\pi=\left(d_i\right)_{i=1}^\ell$,
for several values of $g$ and $h$. We obtain explicit formulae
of arithmetic nature in terms of the local degrees $d_i$.
Our proofs employ a combinatorial method based on Grothendieck's \emph{dessins d'enfant}.

\smallskip

\end{abstract}

\noindent
In this introduction we describe the enumeration problem faced in the present
paper and the situations in which we solve it.

\paragraph{Surface branched covers} A surface branched cover is a
map $$f:\Sigmatil\to\Sigma$$
where $\Sigmatil$ and $\Sigma$ are closed and connected surfaces and $f$ is locally modeled
on maps of the form
$$(\matC,0)\ni z\mapsto z^m\in(\matC,0).$$
If $m>1$ the point
$0$ in the target $\matC$ is called a \emph{branching point},
and $m$ is called the local degree at the point $0$ in the source $\matC$.
There are finitely many branching points, removing which, together
with their pre-images, one gets a genuine cover of some degree $d$.
If there are $n$ branching points, the local degrees at the points
in the pre-image of the $j$-th one form a partition $\pi_j$ of $d$ of some
length $\ell_j$, and the following Riemann-Hurwitz relation holds:
$$\chi\left(\Sigmatil\right)-(\ell_1+\ldots+\ell_n)=d\left(\chi\left(\Sigma\right)-n\right).$$
Let us now call \emph{branch datum} a 5-tuple
$$\left(\Sigmatil,\Sigma,d,n,\pi_1,\ldots,\pi_n\right)$$
and let us say it is \emph{compatible} if it satisfies the Riemann-Hurwitz relation.
(For a non-orientable $\Sigmatil$ and/or $\Sigma$ this relation
should actually be complemented with certain other necessary conditions,
but we restrict to an orientable $\Sigma$ in this paper, so we do not
spell out these conditions here.)

\paragraph{The Hurwitz problem}
The very old \emph{Hurwitz problem} asks which compatible branch data are
\emph{realizable} (namely, associated to some existing surface branched cover)
and which are \emph{exceptional} (non-realizable).
Several partial solutions
to this problem have been obtained over the time, and we quickly
mention here the fundamental~\cite{EKS}, the survey~\cite{Bologna}, and
the more recent~\cite{Pako, PaPe, PaPebis, CoPeZa, SongXu}.
In particular, for an orientable $\Sigma$ the problem has been shown
to have a positive solution whenever $\Sigma$ has positive genus.
When $\Sigma$ is the sphere $S$, many realizability and exceptionality
results have been obtained (some of experimental nature), but the general
pattern of what data are realizable remains elusive. One guiding
conjecture in this context is that \emph{a compatible branch datum is always
realizable if its degree is a prime number}. It was actually shown in~\cite{EKS}
that proving this conjecture in the special case of $3$ branching
points would imply the general case. This is why many efforts have
been devoted in recent years to investigating the realizability
of compatible branch data with base surface $\Sigma$ the sphere $S$ and having $n=3$
branching points. See in particular~\cite{PaPe, PaPebis} for some evidence
supporting the conjecture.

\paragraph{Hurwitz numbers}
Two branched covers
$$f_1:\Sigmatil\to\Sigma\qquad f_2:\Sigmatil\to\Sigma$$
are said to be \emph{weakly equivalent} if there exist homeomorphisms $\gtil:\Sigmatil\to\Sigmatil$
and $g:\Sigma\to\Sigma$
such that $f_1\compo\gtil=g\compo f_2$, and \emph{strongly equivalent} if
the set of branching points in $\Sigma$ is fixed once and forever and
one can take $g=\textrm{id}_\Sigma$.
The \emph{(weak or strong) Hurwitz number} of a compatible
branch datum is the number of (weak or strong) equivalence classes of branched covers
realizing it. So the Hurwitz problem can be rephrased as the question whether a
Hurwitz number is positive or not (a weak Hurwitz number can be smaller
than the corresponding strong one, but they can only vanish simultaneously).
Long ago Mednykh in~\cite{Medn1, Medn2} gave some formulae for the computation
of the strong Hurwitz numbers,
but the actual implementation of these formulae is rather elaborate in general.
Several results were also obtained in more recent years in~\cite{GKL, KM1, KM2, KML, MSS}.

\paragraph{Computations}
In this paper we consider branch data of the form
$$\leqno{(\heartsuit)}
\left(\Sigmatil,\Sigma=S,d=2k,n=3,
(2,\ldots,2),(2h+1,1,2,\ldots,2),\pi=\left(d_i\right)_{i=1}^\ell\right)$$
for $h\geqslant0$.  A direct computation shows that such a datum is compatible
for $h\geqslant 2g$, where $g$ is the genus of $\Sigmatil$, and $\ell=h-2g+1$.
We compute the weak Hurwitz number of the datum for the values of
$g,2h+1,\ell$ shown in boldface in Table~\ref{roadmap:tab}.
\begin{table}
\begin{center}
\begin{tabular}{c||c|c|c|c}
             &   $g=0$                 &   $g=1$                 &     $g=2$               &     $g=3$      \\ \hline\hline
$2h+1=1$     &   $\ell$\textbf{\,=\,1} &   ---                   &     ---                 &     ---        \\ \hline
$2h+1=3$     &   $\ell$\textbf{\,=\,2} &   ---                   &     ---                 &     ---        \\ \hline
$2h+1=5$     &   $\ell$\textbf{\,=\,3} &   $\ell$\textbf{\,=\,1} &     ---                 &     ---        \\ \hline
$2h+1=7$     &   $(\ell=4)$            &   $\ell$\textbf{\,=\,2} &     ---                 &     ---        \\ \hline
$2h+1=9$     &   $(\ell=5)$            &   $(\ell=3)$            &   $\ell$\textbf{\,=\,1} &     ---        \\ \hline
$2h+1=11$    &   $(\ell=6)$            &   $(\ell=4)$            &   $(\ell=2)$            &     ---        \\ \hline
$2h+1=13$    &   $(\ell=7)$            &   $(\ell=5)$            &   $(\ell=3)$            &     $(\ell=1)$
\end{tabular}
\end{center}
\mycap{Values of $g,h,\ell$ giving compatible data $(\heartsuit)$. \label{roadmap:tab}}
\end{table}
More values could be obtained,
including for instance those within parentheses in the table, using the same techniques as we
employ below, but the complication of the topological and
combinatorial situation grows very rapidly, and the arithmetic
formulae giving the weak Hurwitz numbers are likely to be rather intricate
for larger values of $g$ and $h$.

For brevity we will henceforth denote by $\nu$ the number of weakly inequivalent realizations of $(\heartsuit)$
for any given values of $g$ and $h$.

\begin{thm}\label{genus0:thm}
For $g=0$ and $0\leqslant h\leqslant 2$ there hold:
\begin{itemize}
\item For $h=0$, whence $\ell=1$ and $\pi=(2k)$, we always have $\nu=1$;
\item For $h=1$, whence $\ell=2$ and $\pi=(p,2k-p)$ with $p\leqslant k$,
we have $\nu=1$ for $p<k$, and $\nu=0$ for $p=k$;
\item For $h=2$, whence $\ell=3$, we have:
\begin{itemize}
\item[(i)] $\nu=0$ if $k=3m$ and $\pi=(2m,2m,2m)$;
\item[(ii)] $\nu=0$ if $k=2m$ and $\pi=(2m,m,m)$;
\item[(iii)] $\nu=1$ if $\pi=(2t,k-t,k-t)$ with
\begin{itemize}
\item[(a)] $1\leqslant t<\frac k3$, or
\quad (b) $\frac k3<t<\frac k2$, or
\quad (c) $\frac k2<t<k$;
\end{itemize}
\item[(iv)] $\nu=1$ if $\pi=(k,k-r,r)$ with $1\leqslant r<\frac k2$;
\item[(v)] $\nu=2$ if $\pi=(2k-q-r,q,r)$ with $1\leqslant r<\frac k2$ and $r<q<k-r$;
\item[(vi)] $\nu=3$ if $\pi=(2k-q-r,q,r)$ with
\begin{itemize}
\item[(a)] $1\leqslant r<\frac k2$ and $k-r<q<k-\frac r2$, or
\item[(b)] $\frac k2\leqslant r<\frac 23k$ and $r<q<k-\frac r2$.
\end{itemize}
\end{itemize}
\end{itemize}
\end{thm}

\begin{thm}\label{genus1:thm}
For $g=1$ and $2\leqslant h\leqslant 3$ there hold:
\begin{itemize}
\item For $h=2$, whence $\ell=1$, we have $\nu=\left[\frac14(k-1)^2\right]$;
\item For $h=3$, whence $\ell=2$ and $\pi=(p,2k-p)$ with $p\leqslant k$, we have:
\begin{itemize}
\item for $p=k$
$$\nu=\frac12\left[\frac k2\right]\left(\left[\frac k2\right]-1\right)
+\left[\frac14\left[\frac{k-1}2\right]^2\right];$$
\item for $p<k$
\begin{eqnarray*}
  \nu &=& \left[\frac14(p-1)^2\right] +\left[\frac p2\right]\left(\left[\frac p2\right]-1\right)+(k-3)(k-p-1)\\
      & & +\left[\frac14\left(k-\left[\frac{p}2\right]-1\right)^2\right] -\left[\frac{k-p}2\right]
           +\left[\frac14\left[\frac {p-1}2\right]^2\right].
\end{eqnarray*}
\end{itemize}
\end{itemize}
\end{thm}

\begin{thm}\label{genus2:thm}
For $g=2$ and $h=4$, whence $\ell=1$, we have
$$\nu=\frac{k-1}{16}\left(7k^3 - 63k^2 + 197k - 208\right)+
\frac58(5 - 2k)\cdot\left[\frac k2\right].$$
\end{thm}

\section{Computation of weak Hurwitz numbers\\ via dessins d'enfant}\label{DA:sec}
Dessins d'enfant were introduced in~\cite{Groth} (see also~\cite{Cohen}) and
have been already exploited to give partial answers to the Hurwitz problem~\cite{LZ, Bologna}.
Here we show how to employ them to compute weak Hurwitz numbers.
Let us fix until further notice a branch datum
$$\leqno{(\spadesuit)}
\left(\Sigmatil,\Sigma=S,d,n=3,
\pi_1=\left(d_{1i}\right)_{i=1}^{\ell_1},
\pi_2=\left(d_{2i}\right)_{i=1}^{\ell_2},
\pi_3=\left(d_{3i}\right)_{i=1}^{\ell_3}\right).$$
A graph $\Gamma$ is \emph{bipartite} if it has black and white
vertices, and each edge joins black to white. If $\Gamma$ is
embedded in $\Sigmatil$ we call \emph{region} a component $R$ of
$\Sigmatil\setminus\Gamma$, and
\emph{length} of $R$ the number of white (or black) vertices of
$\Gamma$ to which $R$ is incident, with multiplicity
($\partial R$ can be parameterized as a union of locally injective closed
possibly non-simple curves, and the multiplicity of a vertex for
$R$ is the number of times $\partial R$ goes through it).
A pair $(\Gamma,\sigma)$ is called \emph{dessin d'enfant} representing $(\spadesuit)$
if $\sigma\in\permu_3$ and $\Gamma\subset\Sigmatil$ is a bipartite graph
such that:
\begin{itemize}
\item The black vertices of $\Gamma$ have valence $\pi_{\sigma(1)}$;
\item The white vertices of $\Gamma$ have valence $\pi_{\sigma(2)}$;
\item The regions of $\Gamma$ are topological discs and have length $\pi_{\sigma(3)}$.
\end{itemize}
We will also say that $\Gamma$ \emph{represents $(\spadesuit)$ through} $\sigma$.

\begin{rem}
\emph{If $f:\Sigmatil\to S$ is a branched cover matching $(\spadesuit)$
and $\alpha$ is a segment in $S$ with a black and a white end
at the branching points corresponding to $\pi_1$ and $\pi_2$, then
$\left(f^{-1}(\alpha),\textrm{id}\right)$  represents
$(\spadesuit)$, with vertex colours of $f^{-1}(\alpha)$ lifted via $f$.}
\end{rem}

\begin{prop}\label{from:Gamma:to:f:prop}
To a dessin d'enfant $(\Gamma,\sigma)$ representing $(\spadesuit)$
one can associate a branched cover $f:\Sigmatil\to S$
realizing $(\spadesuit)$, well-defined up to equivalence.
\end{prop}

\begin{proof}
We choose distinct points $x_1,x_2,x_3\in S$ and a segment $\alpha$
joining $x_1$ to $x_2$. Then we define $f$ on
$\Gamma$ such that the black vertices are mapped to $x_1$
and the white ones to $x_2$, and $f$ restricted to any edge
is a homeomorphism onto $\alpha$. For each region $R$ of $\Gamma$,
assuming $R$ has length $m$, we fix a point $y$ in $R$ and we extend $f$ so
that:
\begin{itemize}
\item $f$ maps $y$ to $x_3$ locally as
$(\matC,0)\ni z\mapsto z^m\in(\matC,0)$;
\item $f$ is continuous on the closure of $R$;
\item $f$ is a genuine $m:1$ cover from $R\setminus\{y\}$ to $S\setminus(\alpha\cup\{x_3\})$.
\end{itemize}
This construction is of course possible and gives a realization of $(\spadesuit)$
with local degrees $\pi_{\sigma(j)}$ over $x_j$. To see that $f$ is well-defined
up to equivalence we first note that the choice of $x_1,x_2,x_3,\alpha$ is immaterial
up to post-composition with automorphisms of $S$. Now if $f_1,f_2$ are constructed
as described for the same $x_1,x_2,x_3,\alpha$ we can first define an automorphism
$\gtil$ of $\Gamma$ which is the identity on the vertices and given by $f_1^{-1}\compo f_2$
on each edge. Now suppose that to define $f_j$ on a region $R$ of $\Gamma$,
we have chosen the point $y_j\in R$. Then we can define a homeomorphism
$\gtil:R\setminus\{y_2\}\to R\setminus\{y_1\}$ so that $f_1\compo\gtil =f_2$ on $R\setminus\{y_2\}$.
Setting $\gtil(y_2)=y_1$ and patching $\gtil$ with that previously defined on $\Gamma$ we get
the desired equality $f_1\compo\gtil =f_2$ on the whole of $\Sigmatil$.\end{proof}

We define an equivalence relation $\sim$ on dessins d'enfant generated by:
\begin{itemize}
\item $(\Gamma_1,\sigma_1)\sim(\Gamma_2,\sigma_2)$ if $\sigma_1=\sigma_2$ and
there is an automorphism $\gtil:\Sigmatil\to\Sigmatil$ such that
$\Gamma_1=\gtil\left(\Gamma_2\right)$ matching colours;
\item $(\Gamma_1,\sigma_1)\sim(\Gamma_2,\sigma_2)$ if $\sigma_1=\sigma_2\compo(1\,2)$ and
$\Gamma_1=\Gamma_2$ as a set but with vertex colours switched;
\item $(\Gamma_1,\sigma_1)\sim(\Gamma_2,\sigma_2)$ if $\sigma_1=\sigma_2\compo(2\,3)$ and
$\Gamma_1$ has the same black vertices as $\Gamma_2$ and
for each region $R$ of $\Gamma_2$ we have that $R\cap\Gamma_1$ consists
of one white vertex and disjoint edges joining this vertex with the black vertices
on the boundary of $R$ (see Fig.~\ref{DAmove:fig} for an example).
\begin{figure}
    \begin{center}
    \includegraphics[scale=.6]{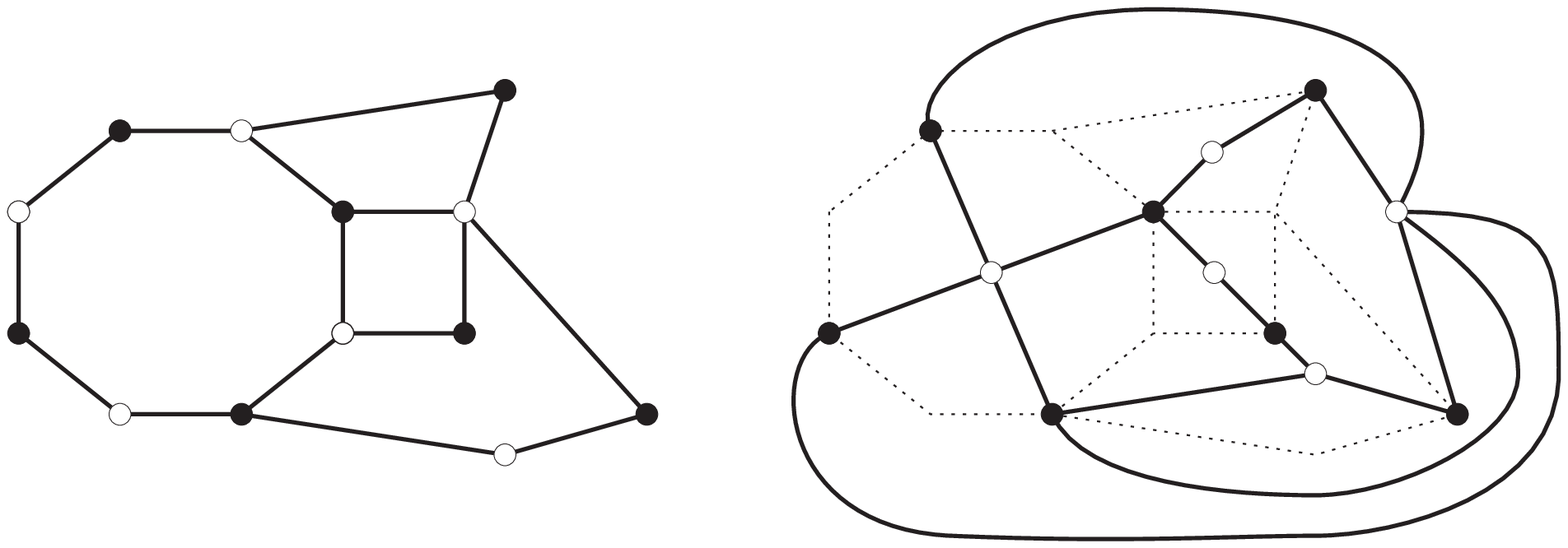}
    \end{center}
\mycap{Dessins d'enfant representing equivalent branched covers. \label{DAmove:fig}}
\end{figure}
\end{itemize}

\begin{thm}\label{equiv:Gamma:for:equiv:f:thm}
The branched covers associated to dessins d'enfant $(\Gamma_1,\sigma_1)$ and
$(\Gamma_2,\sigma_2)$
as in Proposition~\ref{from:Gamma:to:f:prop} are equivalent if and only if
$(\Gamma_1,\sigma_1)\sim(\Gamma_2,\sigma_2)$.
\end{thm}

\begin{proof}
We begin with the ``if'' part. It is enough to prove that the branched covers $f_1$ and $f_2$
associated to $(\Gamma_1,\sigma_1)$ and $(\Gamma_2,\sigma_2)$ are equivalent for
the three instances of $\sim$ generating it:

\begin{itemize}
\item If $\sigma_1=\sigma_2$ and $\Gamma_1=\gtil\left(\Gamma_2\right)$ and
$f_1$ is associated to $(\Gamma_1,\sigma_1)$, then $f_1\compo\gtil$ is associated to
$(\Gamma_2,\sigma_2)$, so $f_1$ is equivalent to $f_2$;

\item If $\sigma_1=\sigma_2\compo(1\,2)$ and $\Gamma_1$ is $\Gamma_2$ with
colours switched, recall that the construction of
$f_1$ and $f_2$ requires the choice of $x_1,x_2,x_3,\alpha$ in $S$.
Choose an automorphism $g$ of $S$ that fixes $x_3$, switches $x_1$ and $x_2$ and
leaves $\alpha$ invariant. Then we see that $g\compo f_2$ is associated to
$(\Gamma_1,\sigma_1)$, so $f_1$ is equivalent to $f_2$;

\item If $\sigma_1=\sigma_2\compo(2\,3)$ and $\Gamma_1,\Gamma_2$ are related as
described above, choose an automorphism $g$ of $S$ that fixes $x_1$ and switches $x_2$ and $x_3$
(so $g(\alpha)$ is a segment joining $x_1$ to $x_3$). Then we see that $g\compo f_2$ is associated to
$(\Gamma_1,\sigma_1)$, so $f_1$ is equivalent to $f_2$.

\end{itemize}

Turning to the ``only if'' part, suppose that
the branched covers $f_1$ and $f_2$
associated to $(\Gamma_1,\sigma_1)$ and $(\Gamma_2,\sigma_2)$ are equivalent,
namely $f_1\compo\gtil=g\compo f_2$ for some automorphisms $\gtil$ and $g$ of $\Sigmatil$ and $S$.
Using $\sim$ we can reduce to the case where $\sigma_1=\sigma_2=\textrm{id}$.
Assuming $f_1,f_2$ have been constructed using the same $x_1,x_2,x_3,\alpha$,
we may still have that $g$ permutes $x_1,x_2,x_3$ non-trivially if
the triple $\pi_1,\pi_2,\pi_3$ contains repetitions, but up to replacing
$f_2$ by a suitable $h\compo f_2$ for an automorphism $h$ of $S$, we can
suppose $g$ fixes $x_1,x_2,x_3,\alpha$. Then $\Gamma_j=f_j^{-1}(\alpha)$
with black vertices over $x_1$ and white over $x_2$, so $\Gamma_1=\gtil(\Gamma_2)$
matching colours.\end{proof}

When the partitions $\pi_1,\pi_2,\pi_3$ in the branch datum $(\spadesuit)$
are pairwise distinct, to compute the corresponding weak Hurwitz number one can
stick to dessins d'enfant representing the datum through the identity, namely
one can list up to automorphisms of $\Sigmatil$ the bipartite graphs
with black and white vertices of valence $\pi_1$ and $\pi_2$ and
discal regions of length $\pi_3$. When the partitions are not distinct, however,
it is essential to take into account the other moves generating $\sim$, see
Fig.~\ref{nonsymexample:fig}.
\begin{figure}
    \begin{center}
    \includegraphics[scale=.6]{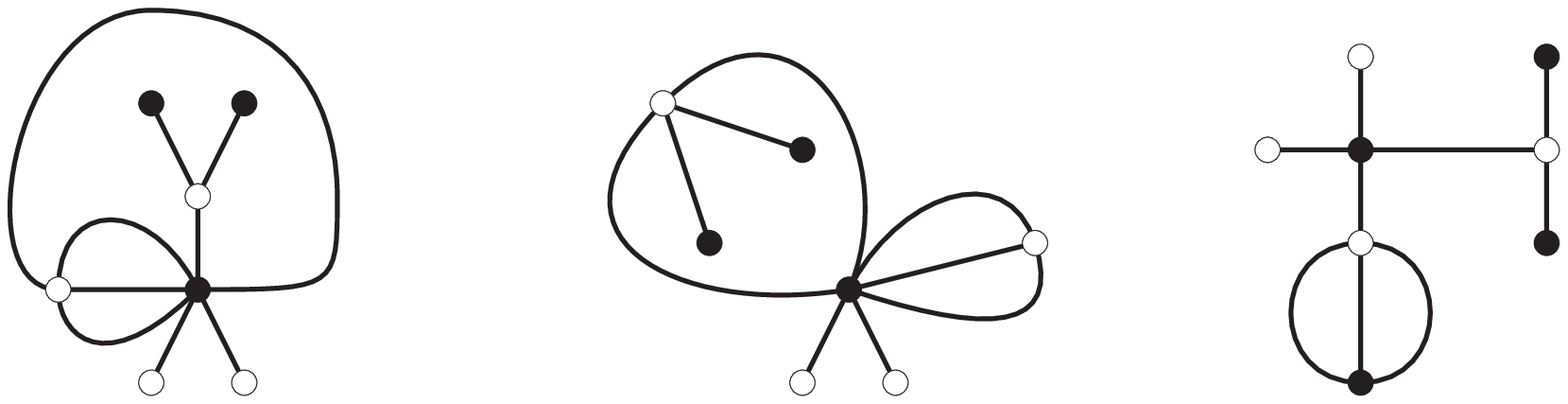}
    \end{center}
\mycap{Call $\Gamma_1,\Gamma_2,\Gamma_3$ the dessins d'enfant appearing
in this picture. Then $\Gamma_1$ and $\Gamma_2$ both represent the datum $(S,S,9,3,(7,1,1),(4,3,1,1),(4,3,1,1))$ through
the identity and through $(2\,3)$. They are not even abstractly homeomorphic as uncoloured graphs, but they define equivalent branched coverings,
because they are obtained from each other by the last move generating $\sim$. Note also that $\Gamma_3$ is obtained
from $\Gamma_1$ by applying the second move generating $\sim$ and then the last move, so $\Gamma_3$ also
represents the same datum through $(1\,3)$ and through $(1\,2\,3)$.
\label{nonsymexample:fig}}
\end{figure}

\paragraph{Relevant data and repeated partitions}
We now specialize again to a branch datum of the form $(\heartsuit)$.
We will compute its weak Hurwitz number $\nu$ by enumerating up to
automorphisms of $\Sigmatil$ the dessins d'enfant $\Gamma$
representing it through the identity, namely
the bipartite graphs $\Gamma$ with black vertices of
valence $(2,\ldots,2)$, white vertices of valence
$(2h+1,1,2,\ldots,2)$, and discal regions of
length $\pi$. Two remarks are in order:
\begin{itemize}
\item In all the pictures we will only draw the two white
vertices of $\Gamma$ of valence $(2h+1,1)$, and
we will decorate an edge of $\Gamma$ by an integer $a\geqslant1$
to understand that the edge contains $a$ black and $a-1$ white valence-2 vertices;
\item Enumerating these dessins d'enfant $\Gamma$ up to automorphisms of $\Sigmatil$
already gives the right value of $\nu$ except if two of the
partitions of $d$ in the datum coincide.
\end{itemize}

\begin{prop}
In a branch datum of the form $(\heartsuit)$ two of the partitions of $d$ coincide precisely in
the following cases:
\begin{itemize}
\item $g=0,\ h\geqslant 0,\ k=h+1$, with partitions $(2,\ldots,2),(2h+1,1),(2,\ldots,2)$;
\item Any $g,\ h\geqslant 2g+1,\ k=2h-2g$, with partitions $(2,\ldots,2),(2h+1,1,2,\ldots,2),(2h+1,1,2,\ldots,2)$.
\end{itemize}
\end{prop}

\begin{proof}
The lengths of the partitions $\pi_1,\pi_2,\pi$ in $(\heartsuit)$ are $\ell_1=k$, $\ell_2=k-h+1$ and
$\ell=h+1-2g$. We can never have $\pi_1=\pi_2$. Since $k\geqslant h+1$ we can have $\ell_1=\ell$ only
if $g=0$ and $k=h+1$, whence the first listed item. We can have $\ell_2=\ell$ only for $k=2h-2g$,
whence $h\geqslant 2g+1$ and the data in the second listed item.
\end{proof}

This result implies that the data $(\heartsuit)$ relevant to Theorems~\ref{genus0:thm} to~\ref{genus2:thm} and
containing repetitions  are precisely
$$(S,S,2,3,(2),(1,1),(2))\qquad (S,S,4,3,(2,2),(3,1),(2,2))$$
$$(S,S,6,3,(2,2,2),(5,1),(2,2,2))\qquad  (S,S,4,3,(2,2),(3,1),(3,1))$$
$$(S,S,8,3,(2,2,2,2),(5,1,2),(5,1,2))\qquad (T,S,8,3,(2,2,2,2),(7,1),(7,1))$$
(where $T$ is the torus) for which we easily have $\nu=1$ in the first case, and
$\nu=0$ in the second and third one by the \emph{very even data} criterion of~\cite{PaPebis}.
The last three cases will be taken into account in Sections~\ref{genus0:sec} and~\ref{genus1:sec}.

\section{Genus 0}\label{genus0:sec}
In this section we prove Theorem~\ref{genus0:thm}, starting
from the very easy cases $h=0$ and $h=1$, for which there is only
one homeomorphism type of relevant graph and only one embedding
in $S$, as shown in Fig.~\ref{genus0easy:fig} --- here and below
the notion of $(2h+1,1)$ \emph{graph} abbreviates \emph{a graph with vertices
of valence} $(2h+1,1)$.
\begin{figure}
    \begin{center}
    \includegraphics[scale=.6]{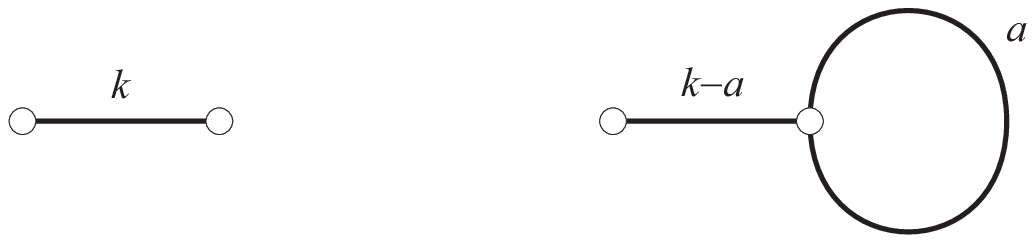}
    \end{center}
\mycap{The $(1,1)$ and $(3,1)$ graphs in $S$. \label{genus0easy:fig}}
\end{figure}
The first graph gives a unique realization of $(2k)$ as $\pi$.
The second graph with $a=p$ gives a unique realization of $\pi=(p,2k-p)$ for $p<k$,
while $(k,k)$ is exceptional. Note that a single graph emerges for the
realization of the case with repeated partitions
$(S,S,4,3,(2,2),(3,1),(3,1))$, so its realization is \emph{a fortiori} unique
up to equivalence (and it is also immediate to check that the move of Fig.~\ref{DAmove:fig}
leads this graph to itself).

For the case $h=2$, whence $\ell=3$,
we first check that the cases listed in the statement cover all the possibilities for $\pi=(p,q,r)$:
\begin{itemize}
\item If $p=q=r$ we have $3p=2k$, so $k=3m$ and $p=2m$, whence case (i);
\item If $p\neq q=r$ we have $p=2k-2q$, so $p=2t$ and $q=k-t$, with $k-t\neq 2t$, so $t\neq \frac k3$,
and $0<t<k$; then either $t=\frac k2$, so $k=2m$ and case (ii), or one of the
subcases (a), (b) or (c) of (iii);
\item If $p>q>r$ we have $p=2k-q-r$; since $q>r$ we have $2k-2r>2k-q-r>q>r$,
whence $1\leqslant r<\frac23k$ and $r<q<k-\frac r2$; for $r<\frac k2$ by comparing
$q$ with $k-r$ we get case (iv) or (v) or (vi)-(a), while for $r\geqslant \frac k2$
we get (vi)-(b).
\end{itemize}

Even if this is not strictly necessary, we also provide
in Tables~\ref{Tk6:tab} and~\ref{Tk7:tab} two examples
of an application of Theorem~\ref{genus0:thm} for $h=2$, proving
that each  listed case is non-empty for suitable $k$.
The tables include a description of the realizations with
notation that will be introduced in the proof.
\begin{table}
\begin{center}
\begin{tabular}{c||l|c|l}
$\pi$       &   Case
                &   $\nu$   &   Realizations \\ \hline\hline
(10,1,1)    & (iii)-(c) \tiny{$\frac k2=3<t=5<k=6$}
                & 1 & \small{I(4,1,1)}  \\ \hline
(9,2,1)     & (v) \tiny{$\begin{array}{l}1\leqslant r=1<\frac k2 =3\\ r=1<q=2<k-r=5\end{array}$}
                & 2 & \small{I(3,2,1), I\!I(4,1,1)}  \\ \hline
(8,3,1)     & (v) \tiny{$\begin{array}{l}1\leqslant r=1<\frac k2 =3\\ r=1<q=3<k-r=5\end{array}$}
                & 2 & \small{I(2,3,1), I\!I(3,2,1)}  \\ \hline
(8,2,2)     & (iii)-(c) \tiny{$\frac k2=3<t=4<k=6$}
                & 1 & \small{I(2,2,2)}  \\ \hline
(7,4,1)     & (v) \tiny{$\begin{array}{l}1\leqslant r=1<\frac k2 =3\\ r=1<q=4<k-r=5\end{array}$}
                & 2 & \small{I(1,4,1), I\!I(2,3,1)}  \\ \hline
(7,3,2)     & (v) \tiny{$\begin{array}{l}1\leqslant r=2<\frac k2 =3\\ r=2<q=3<k-r=4\end{array}$}
                & 2 & \small{I(1,3,2),\ I\!I(3,1,2)}  \\ \hline
(6,5,1)     & (iv) \tiny{$1\leqslant r=1<\frac k2 =3$}
                & 1 & \small{I\!I(1,4,1)}  \\ \hline
(6,4,2)     & (iv) \tiny{$1\leqslant r=2<\frac k2 =3$}
                & 1 & \small{I\!I(2,2,2)}  \\ \hline
(6,3,3)     & (ii) \tiny{$m=3$}
                & 0 &   \\ \hline
(5,5,2)     & (iii)-(a) \tiny{$1\leqslant t=1<\frac k3=2$}
                & 1 & \small{I\!I(1,3,2)}  \\ \hline
(5,4,3)     & (vi)-(b) \tiny{$\begin{array}{l}\frac k2=3\leqslant r=3<\frac23k=4\\ r=3<q=4<k-\frac r2=4.5\end{array}$}
                & 3 & \small{I\!I(1,1,4),\ I\!I(2,1,3),\ I\!I(1,2,3)}  \\ \hline
(4,4,4)     & (i) \tiny{$m=2$}
                & 0 &
\end{tabular}
\end{center}
\mycap{The case $k=6$. \label{Tk6:tab}}
\end{table}

\begin{table}
\begin{center}
\begin{tabular}{c||l|c|l}
$\pi$       &   Case
                &   $\nu$   &   Realizations \\ \hline\hline
(12,1,1)    & (iii)-(c)   \tiny{$\frac k2=3.5<t=6<k=7$}
                & 1  & \small{I(5,1,1)}  \\ \hline
(11,2,1)    & (v)  \tiny{$\begin{array}{l} 1\leqslant r=1<\frac k2=3.5\\ r=1<q=2<k-r=6 \end{array}$}
                & 2  & \small{I(4,2,1),\ I\!I(5,1,1)}  \\ \hline
(10,3,1)    & (v)  \tiny{$\begin{array}{l} 1\leqslant r=1<\frac k2=3.5\\ r=1<q=3<k-r=6 \end{array}$}
                & 1  & \small{I(3,3,1),\ I\!I(4,2,1)}  \\ \hline
(10,2,2)    & (iii)-(c) \tiny{$\frac k2=3.5<t=5<k=7$}
                & 1  &  \small{I(3,2,2)} \\ \hline
(9,4,1)     & (v)  \tiny{$\begin{array}{l} 1\leqslant r=1<\frac k2=3.5\\ r=1<q=4<k-r=6 \end{array}$}
                & 2  &  \small{I(2,4,1),\ I\!I(3,3,1)} \\ \hline
(9,3,2)     & (v)  \tiny{$\begin{array}{l} 1\leqslant r=2<\frac k2=3.5\\ r=2<q=3<k-r=5 \end{array}$}
                & 2  &  \small{I(2,3,2),\ I\!I(4,1,2)} \\ \hline
(8,5,1)     & (v)  \tiny{$\begin{array}{l} 1\leqslant r=1<\frac k2=3.5\\ r=1<q=5<k-r=6 \end{array}$}
                & 2  & \small{I(1,5,1),\ I\!I(2,4,1)}  \\ \hline
(8,4,2)     & (v) \tiny{$\begin{array}{l} 1\leqslant r=2<\frac k2=3.5\\ r=2<q=4<k-r=5 \end{array}$}
                & 2  &  \small{I(1,4,2),\ I\!I(3,2,2)} \\ \hline
(8,3,3)     & (iii)-(c) \tiny{$\frac k2=3.5<t=4<k=7$}
                & 1  &  \small{I(1,3,3)} \\ \hline
(7,6,1)     & (iv)  \tiny{$1\leqslant r=1<\frac k2=3.5$}
                & 1  &  \small{I\!I(1,5,1)} \\ \hline
(7,5,2)     & (iv)  \tiny{$1\leqslant r=2<\frac k2=3.5$}
                & 1  &  \small{I\!I(2,3,2)} \\ \hline
(7,4,3)     & (iv)  \tiny{$1\leqslant r=3<\frac k2=3.5$}
                & 1  &  \small{I\!I(3,1,3)} \\ \hline
(6,6,2)     & (iii)-(a) \tiny{$1\leqslant t=1<\frac k3=2.\overline{3}$}
                & 1  &  \small{I\!I(1,4,2)} \\ \hline
(6,5,3)     & (vi)-(a)  \tiny{$\begin{array}{l} 1\leqslant r=3<\frac k2=3.5\\ k-r=4<q=5<k-\frac r2=5.5     \end{array}$}
                & 3  &  \small{I\!I(1,1,5),\ I\!I(1,3,3),\ I\!I(2,2,3)} \\ \hline
(6,4,4)     & (iii)-(b) \tiny{$\frac k3=2.\overline{3}<t=3<\frac k2=3.5$}
                & 1  &  \small{I\!I(1,2,4)}\\ \hline
(5,5,4)     & (iii)-(a) \tiny{$1\leqslant t=2<\frac k3=2.\overline{3}$}
                & 1  &  \small{I\!I(2,1,4)}
\end{tabular}
\end{center}
\mycap{The case $k=7$. \label{Tk7:tab}}
\end{table}

\medskip

Let us now get to the actual proof.  There are only two inequivalent
embeddings in $S$ of the $(5,1)$ graph, shown in Fig.~\ref{51genus0:fig}
\begin{figure}
    \begin{center}
    \includegraphics[scale=.6]{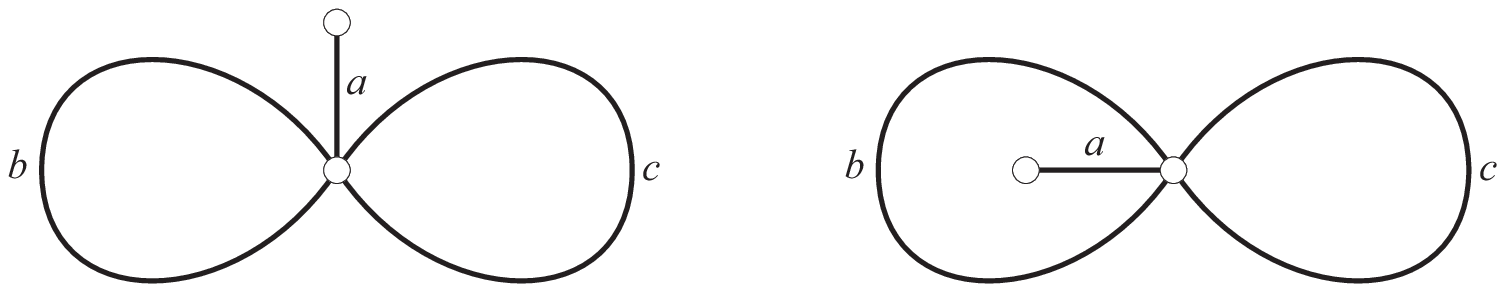}
    \end{center}
\mycap{The $(5,1)$ graphs in $S$: $\primo(a,b,c)$ on the left and $\secon(a,b,c)$ on the right. \label{51genus0:fig}}
\end{figure}
with edges decorated as explained at the end of Section~\ref{DA:sec}, and
denoted by $\primo(a,b,c)$ and $\secon(a,b,c)$, realizing respectively
$\pi=(2a+b+c,b,c)$ and $\pi=(2a+b,b+c,c)$.
Note that there
there is no automorphism taking $\secon$ to itself, while there is one
taking $\primo(a,b,c)$ to $\primo(a,c,b)$.

Given $k$ and $\pi=(p,q,r)$ with $p+q+r=2k$
we must now count how many realizations we have of $\pi$ as $(2a+b+c,b,c)$ up
to switching $b$ and $c$, or as $(2a+b,b+c,c)$. Of course if $p=q=r$ there is no
realization, whence case (i) of the statement.

If $p=q\neq r$, so $\pi=(2t,k-t,k-t)$ with $t\neq \frac k3$, we can realize it via $\primo(a,b,c)$ precisely if
$$\sistema{2a+b+c=2t\\ b=c=k-t}\quad\Leftrightarrow\quad\sistema{a=2t-k\\ b=c=k-t}\quad \textrm{for}\ \frac k2<t<k$$
whereas we can realize it via $\secon(a,b,c)$ in the following cases:
$$\sistema{2a+b=b+c=k-t\\ c=2t}\quad\Leftrightarrow\quad\sistema{a=t\\ b=k-3t\\ c=2t}\quad \textrm{for}\ 1\leqslant t<\frac k3$$
$$\sistema{2a+b=c=k-t\\ b+c=2t}\quad\Leftrightarrow\quad\sistema{a=k-2t\\ b=3t-k\\ c=k-t}\quad \textrm{for}\ \frac k3<t<\frac k2.$$
We have found the three disjoint instances (a), (b), (c) of case (iii), and we have not found
realizations with $t=\frac k2$, whence (ii).

\medskip

Turning to the case $p>q>r$, so $p=2k-q-r$ with $1\leqslant r<\frac23k$ and $r<q<k-\frac r2$,
the realizations via $\primo(a,b,c)$ are given by
$$\sistema{2a+b+c=2k-q-r\\ b=q\\ c=r}\quad\Leftrightarrow\quad\sistema{a=k-q-r\\ b=q\\ c=r}\quad
\begin{array}{l}
\textrm{for}\ r<q<k-r\\
\textrm{whence}\ 1\leqslant r<\frac k2
\end{array}$$
which gives the first contribution to (v). The realizations via $\secon(a,b,c)$ are instead
$$\sistema{2a+b=2k-q-r\\ b+c=q\\ c=r}\quad\Leftrightarrow\quad\sistema{a=k-q\\ b=q-r\\ c=r}$$
which gives the only contribution to (iv), the second and last contribution to (v), and
the first contribution to (vi)-(a) and (vi)-(b), or
$$\sistema{b+c=2k-q-r\\ 2a+b=q\\ c=r}\quad\Leftrightarrow\quad\sistema{a=q+r-k\\ b=2k-q-2r\\ c=r}
\quad \textrm{for}\ q>k-r$$
(note that $q<2k-2r$ is implied by $q<k-\frac r2$ and $r<\frac23k$), or
$$\sistema{b+c=2k-q-r\\ c=q\\ 2a+b=r}\quad\Leftrightarrow\quad\sistema{a=q+r-k\\ b=2k-2q-r\\ c=q}
\quad \textrm{for}\ q>k-r$$
($2k-2q-r>0$ is equivalent to $q<k-\frac r2$); these realizations both give the final
contributions to (vi)-(a) and (vi)-(b), which are respectively
obtained by comparing $r$ to $k-r$, since we must have $q>\max\{r,k-r\}$.

The last issue is to deal with the datum with repeated
partitions $$(S,S,8,3,(2,2,2,2),(5,1,2),(5,1,2)).$$ This can be realized as $\primo(1,2,1)$ and as $\secon(2,1,1)$, and we must
show that these graphs are not obtained from each other via the move of Fig.~\ref{DAmove:fig}.
This is done in Fig.~\ref{51g0selfdual:fig},
\begin{figure}
    \begin{center}
    \includegraphics[scale=.6]{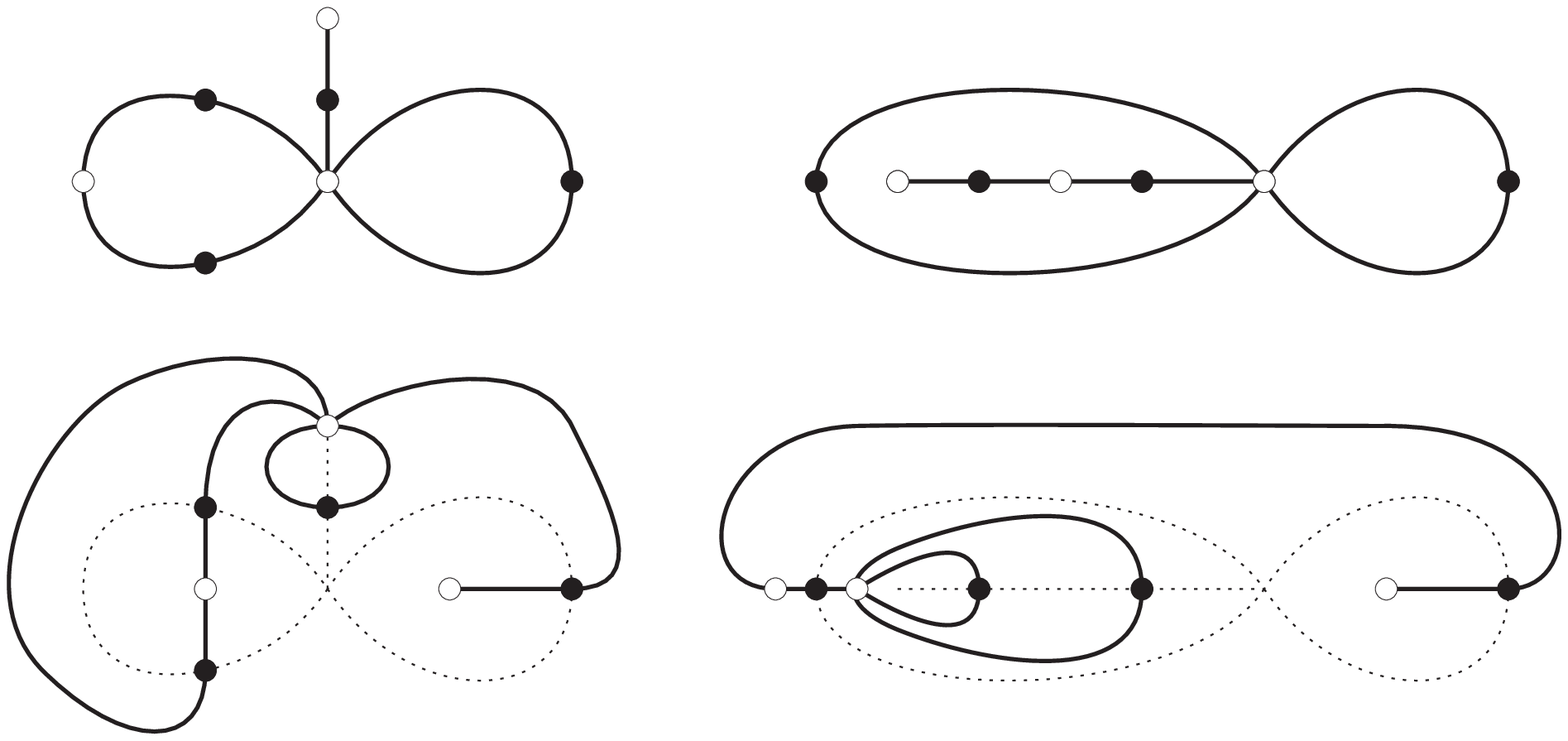}
    \end{center}
\mycap{Graphs mapped to themselves by the move of Fig.~\ref{DAmove:fig}. \label{51g0selfdual:fig}}
\end{figure}
which proves that the move actually takes each of these graphs to itself.

The proof is complete.

\section{Genus 1}\label{genus1:sec}

Let us prove Theorem~\ref{genus1:thm}.
We begin with the rather easy case $h=2$. Up to automorphisms
there is only one embedding in $T$ of a $(5,1)$ graph
with a disc as only region, shown in Fig.~\ref{51genus1:fig}
\begin{figure}
    \begin{center}
    \includegraphics[scale=.6]{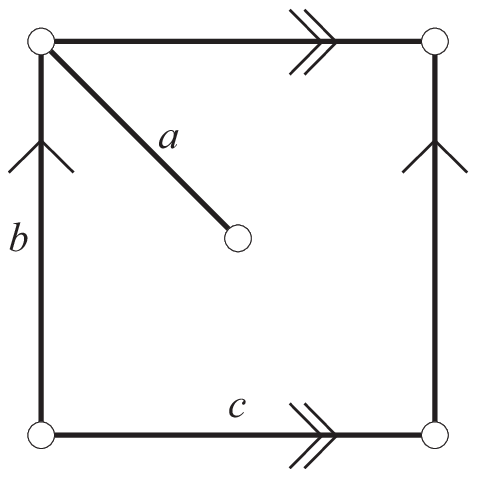}
    \end{center}
\mycap{The $(5,1)$ graph in $T$. \label{51genus1:fig}}
\end{figure}
and subject to the symmetry $b\leftrightarrow c$. So
$$\nu=\#\left\{(a,\{b,c\}):\ a+b+c=k\right\}=\sum_{a=1}^{k-2}\left[\frac{k-a}2\right]=\sum_{j=2}^{k-1}\left[\frac j2\right].$$
For odd and even $k$ one easily gets the expressions $\frac14(k-1)^2$ and $\frac14k(k-2)$ that one can unify
as $\left[\frac14(k-1)^2\right]$.

\bigskip

Turning to $h=3$, so $\ell=2$ and $\pi=(p,2k-p)$ with $1\leqslant p\leqslant k$, we first determine the embeddings
in $T$ of the bouquet $B$ of $3$ circles with two discs as regions. Of course at least a circle of $B$
is non-trivial on $T$, so its complement is an annulus. Then another circle must joint the boundary components
of this annulus, so we can assume two circles of $B$ form a standard meridian-longitude pair on $T$.
Then the possibilities for $B$ are as in Fig.~\ref{bouquetinT:fig}.
\begin{figure}
    \begin{center}
    \includegraphics[scale=.6]{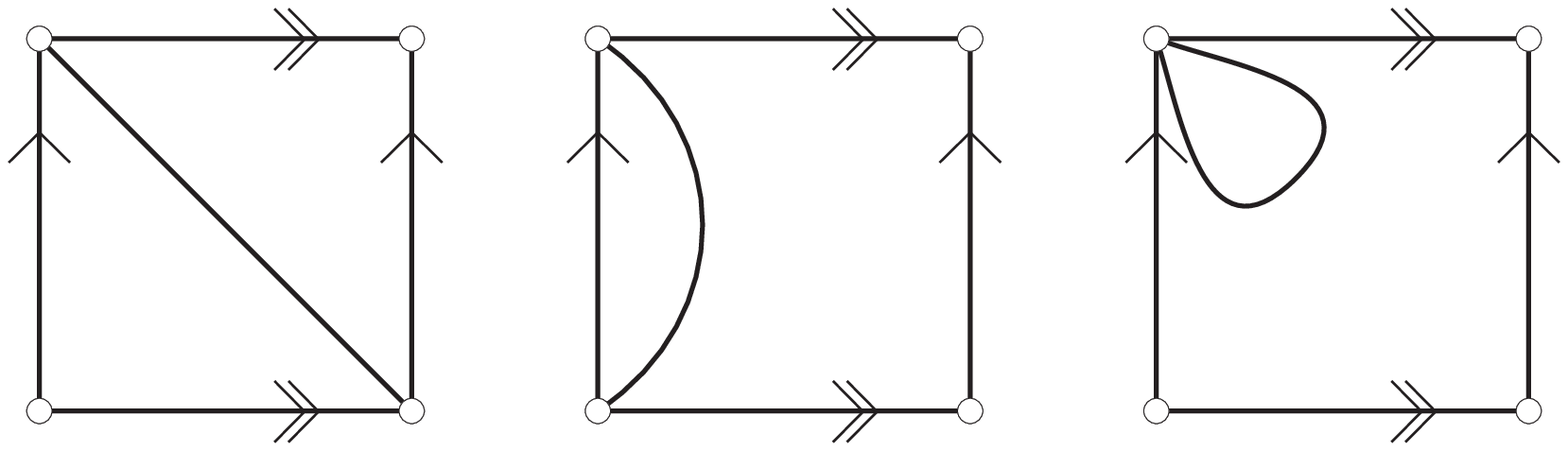}
    \end{center}
\mycap{A bouquet of 3 circles in $T$ with 2 discs as regions. \label{bouquetinT:fig}}
\end{figure}
Note that these embeddings have respectively a $\permu_3\times\matZ/_2$, a $\matZ/_2\times\matZ/_2$ and a
$\matZ/_2$ symmetry. It easily follows that the relevant embeddings in $T$ of the $(7,1)$ graph are
up to symmetry those shown in Fig.~\ref{71embT:fig}.
\begin{figure}
    \begin{center}
    \includegraphics[scale=.6]{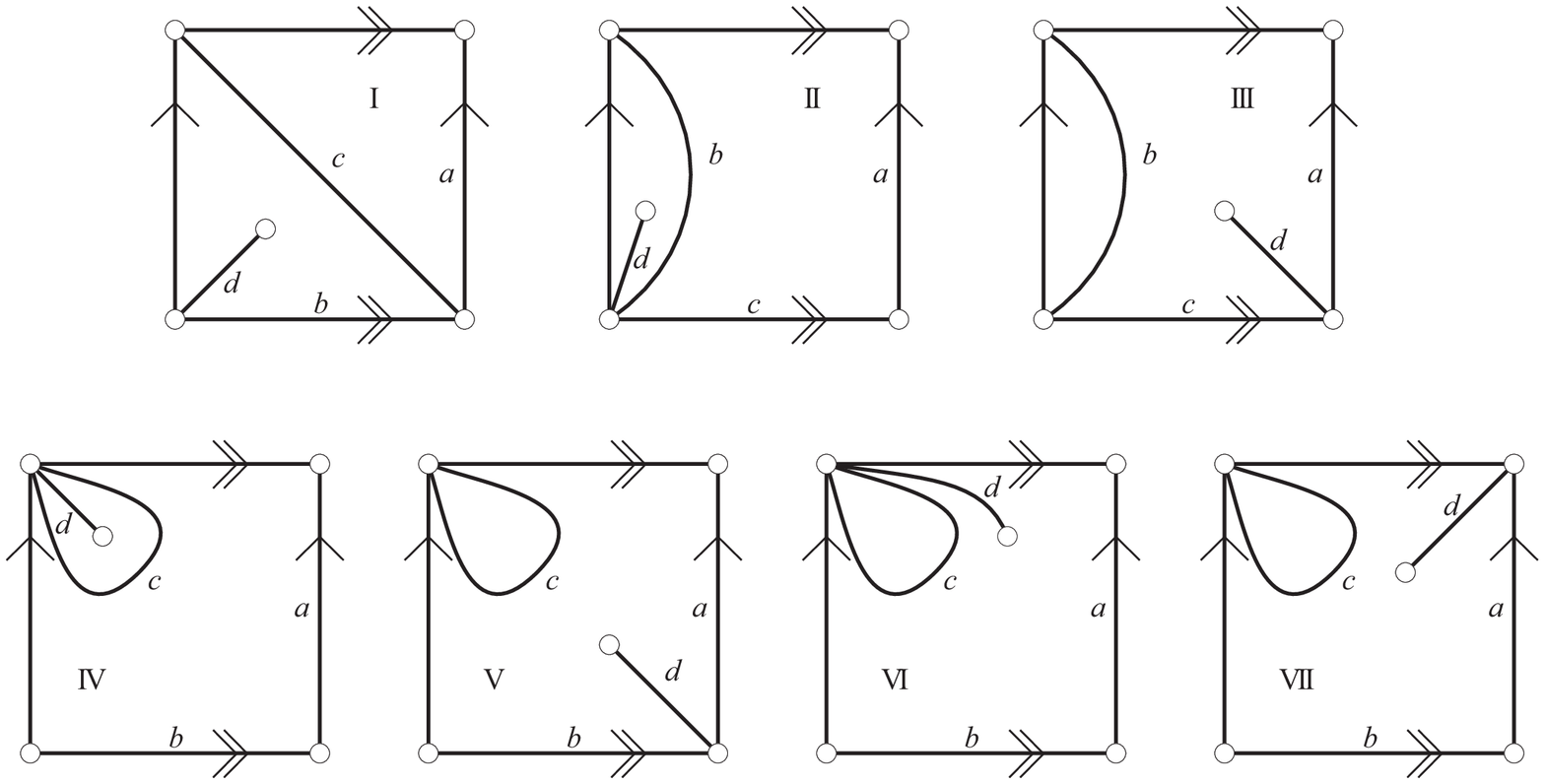}
    \end{center}
\mycap{Embeddings in $T$ of the $(7,1)$ graph with 2 discs as regions. \label{71embT:fig}}
\end{figure}
Note that we have a symmetry switching $a$ and $b$ in cases $\primo,\ \secon,\ \quart,\ \quint$, and no other one.
Moreover $\primo(a,b,c,d)$ realizes $(a+b+c+2d,a+b+c)$,
while $\secon(a,b,c,d)$ realizes $(a+b+2c,a+b+2d)$, then
$\terzo(a,b,c,d)$ realizes $(a+b+2c+2d,a+b)$, next
$\quart(a,b,c,d)$ realizes $(2a+2b+c,c+2d)$ and finally
$\quint(a,b,c,d)$, $\sesto(a,b,c,d)$ and $\setti(a,b,c,d)$ all realize
$(2a+2b+c+2d,c)$. We now count how many different realizations of $\pi=(p,2k-p)$ with
$p\leqslant k$ exist.

\medskip

\noindent $\primo$  \quad $\pi$ is realized if $3\leqslant p<k$ in
$$\#\{(\{a,b\},c):\ a+b+c=p\}=\left[\frac14(p-1)^2\right]$$
ways. Note that the expression gives the right value $0$ also for $p=1,2$.

\medskip

\noindent $\secon$ \quad $\pi$ is realized if $4\leqslant p\leqslant k$. For $p<k$
the number of realizations is
\begin{eqnarray*}
    & & \#\{(\{a,b\},c):\ a+b+2c=p\}+\#\{(\{a,b\},d):\ a+b+2d=p\} \\
    &=& 2\#\{(\{a,b\},c):\ a+b+2c=p\}=2\sum_{c=1}^{[p/2]-1}\left[\frac{p-2c}2\right] \\
    &=& 2\sum_{c=1}^{[p/2]-1}\left(\left[\frac{p}{2}\right]-c\right)
    =2\left(\left[\frac{p}{2}\right]-1\right)\left[\frac{p}{2}\right]-\left(\left[\frac{p}{2}\right]-1\right)\left[\frac{p}{2}\right] \\
    &=& \left(\left[\frac{p}{2}\right]-1\right)\left[\frac{p}{2}\right]
\end{eqnarray*}
(which is correct also for $p=1,2,3$), while for $p=k$ it is $\frac12\left[\frac k2\right]\left(\left[\frac k2\right]-1\right)$.

\medskip

\noindent $\terzo$ \quad $\pi$ is realized if $1<p<k-1$ and the number of realizations is
$$\#\{(a,b):\ a+b=p\}\cdot\#\{(c,d):\ c+d=k-p\}=(p-1)(k-p-1)$$
which works also for $p=1$ and $p=k-1$.

\medskip

\noindent $\quart$ \quad We first consider the case $p<k$. Then
$\pi$ is realized in the following ways:
$$\sistema{c+2d=p\\ 2a+2b+c=2k-p}\quad\Leftrightarrow\quad
\sistema{1\leqslant d\leqslant\left[\frac{p-1}2\right]\\
c=p-2d\\ a+b=k-p-d}\quad
\textrm{for}\ 3\leqslant p<k,$$
$$\sistema{c+2d=2k-p\\ 2a+2b+c=p}\quad\Leftrightarrow\quad
\sistema{k-p+2\leqslant d\leqslant k-1-\left[\frac{p}2\right]\\
c=2k-p-2d\qquad\qquad\qquad\textrm{for}\ 5\leqslant p<k.\\
a+b=p-k+d}$$ So we get for $5\leqslant p<k$ the number
$$\sum_{d=1}^{\left[\frac{p-1}2\right]}
\left[\frac{k-p+d}2\right]+
\sum_{d=k-p+2}^{k-1-\left[\frac p2\right]}
\left[\frac{p-k+d}2\right]$$
but the expression is correct for $1\leqslant p<k$
because both sums vanish for $p=1,2$ and the latter does for $p=3,4$.
An easy computation based on the fact already shown that $\sum_{j=1}^x\left[j/2\right]=\left[\frac14 x^2\right]$
gives the value
$$\left[\frac14\left(k-1-\left[\frac{p}2\right]\right)^2\right] -\left[\frac14(k-p)^2\right]+
\left[\frac14\left[\frac {p-1}2\right]^2\right].$$
For $p=k$ the realizations correspond to
$$\sistema{2a+2b+c=k\\ c+2d=k}\quad\Leftrightarrow\quad
\sistema{1\leqslant d\leqslant\left[\frac{k-1}2\right]\\
c=k-2d\\
a+b=d}$$
so their number is
$$\sum_{d=1}^{\left[\frac{k-1}2\right]}\left[\frac d2\right]=\left[\frac14\left[\frac{k-1}2\right]^2\right].$$

\medskip

\noindent $\quint$ \quad $\pi$ can be realized if $1\leqslant p\leqslant k-3$, via
$$\sistema{2a+2b+c+2d=2k-p\\ c=p}\quad\Leftrightarrow\quad
\sistema{1\leqslant d\leqslant k-p-2\\
a+b=k-p-d\\ c=p}$$
so there are $\left[\frac14(k-p-1)^2\right]$ ways (remember the $a\leftrightarrow b$ symmetry)
and the expression is correct also for $p=k-2$ and $p=k-1$.

\medskip

\noindent $\sesto$,\ $\setti$ \quad The realizations are exactly the same, but we have no symmetry, so there are $\binom{k-p-1}2$
of them, which again is correct for all $p<k$.

\bigskip

The last task before concluding is to take into account the symmetry of the datum
$(T,S,8,3,(2,2,2,2),(7,1),(7,1))$. The above discussion readily implies that its
realizations come from \quint, \sesto, \setti\ with $a=b=c=d=1$, hence from the
dessins d'enfant in the top part of Fig.~\ref{selfdual:fig}.
\begin{figure}
    \begin{center}
    \includegraphics[scale=.6]{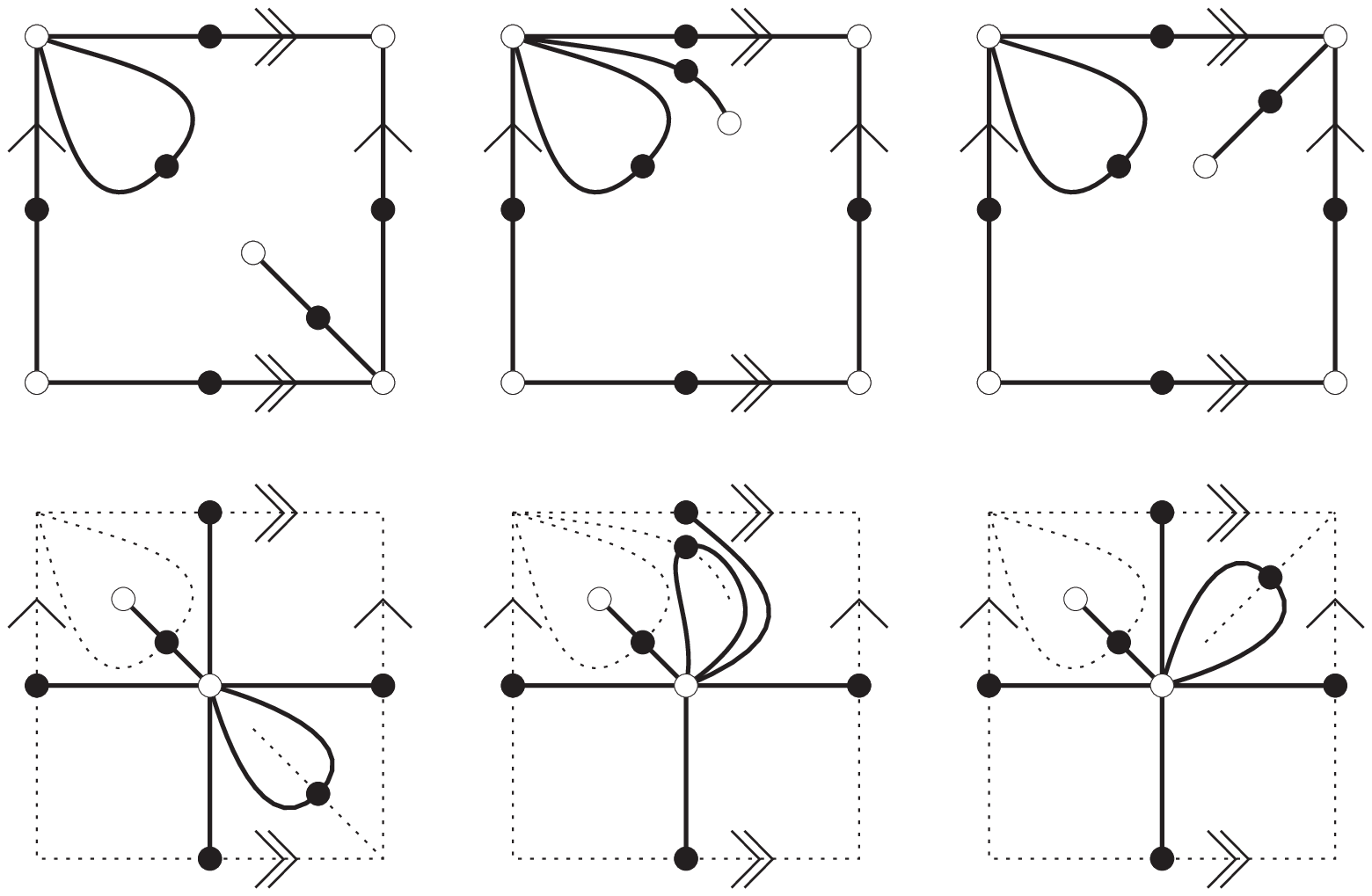}
    \end{center}
\mycap{Moves on dessins d'enfant taking them to themselves. \label{selfdual:fig}}
\end{figure}
In the bottom part of the same figure we show the dessins obtained
by the move that corresponds to the switch of the partitions $(7,1)$.
Since in each case we obtain the same dessin, the move has no effect on the counting.

To conclude we are left to sum all the contributions. For $p=k$ we have the two summands
of the statement.  For $p<k$ we have some simplifications: the contributions from \terzo,\ \sesto,\ \setti\ give
$$(p-1)(k-p-1)+2\cdot\frac12(k-p-1)(k-p-2)=(k-3)(k-p-1)$$
while the second summand from \quart\ and \quint\ gives
$$-\left[\frac14(k-p)^2\right]+\left[\frac14(k-p-1)^2\right]=-\left[\frac{k-p}2\right]$$
and the proof is complete.

\section{Genus 2}\label{genus2:sec}

We now prove Theorem~\ref{genus2:thm}. The topological argument will be
more elaborate than that for Theorem~\ref{genus1:thm}, while the arithmetic part
will be much easier. We denote by $2T$ the genus-$2$ surface and we begin with the following:

\medskip

\noindent
\textsc{Claim}. \emph{Up to automorphisms, there exist precisely $4$ embeddings in $2T$ of the bouquet of $4$ circles
having a disc as only region. They are given by the $1$-skeleton of the realizations of $2T$ as an
octagon $O$ with paired edges shown in Fig.~\ref{4octagons:fig}.}
\begin{figure}
    \begin{center}
    \includegraphics[scale=.6]{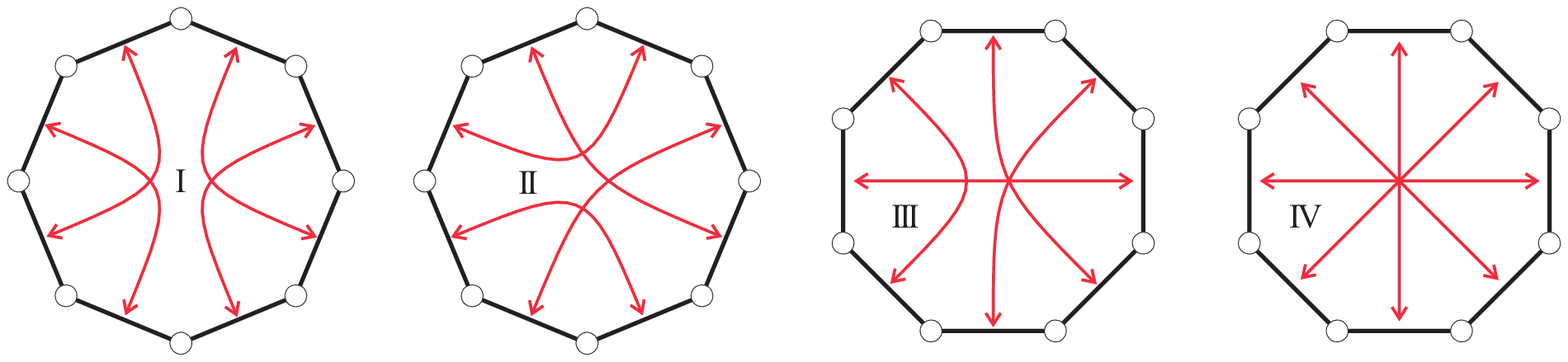}
    \end{center}
\mycap{Inequivalent realizations of $2T$ as an octagon $O$ with paired edges. \label{4octagons:fig}}
\end{figure}

\medskip

\noindent To prove the claim, we note that
an embedding as described always gives a realization of $2T$ as $O$ with paired edges,
so we must discuss how many of these exist. Note that an edge-pairing gives $2T$ precisely if
all the vertices of $O$ are equivalent under it.
We cyclically give colours in $\matZ/_8$ to the edges of $O$, so a pairing consists of 4 pairs $\{\{i_0,i_1\},\{i_2,i_3\},\{i_4,i_5\},\{i_6,i_7\}\}$
with $\{i_0,\ldots,i_7\}=\matZ/_8$. Up to symmetry we suppose
that $i_0=0$ and $i_1$ is the minimal distance between two
paired edges. Up to symmetry we have $i_1\leqslant 4$
and we also have $i_1\geqslant 2$ because if $i_1=1$ the vertex between $0$ and $1$
is equivalent to itself only. So we can suppose $i_2=1$.
The rest of the proof of the claim is pictorially illustrated in Fig.~\ref{octaproof:fig}
\begin{figure}
    \begin{center}
    \includegraphics[scale=.6]{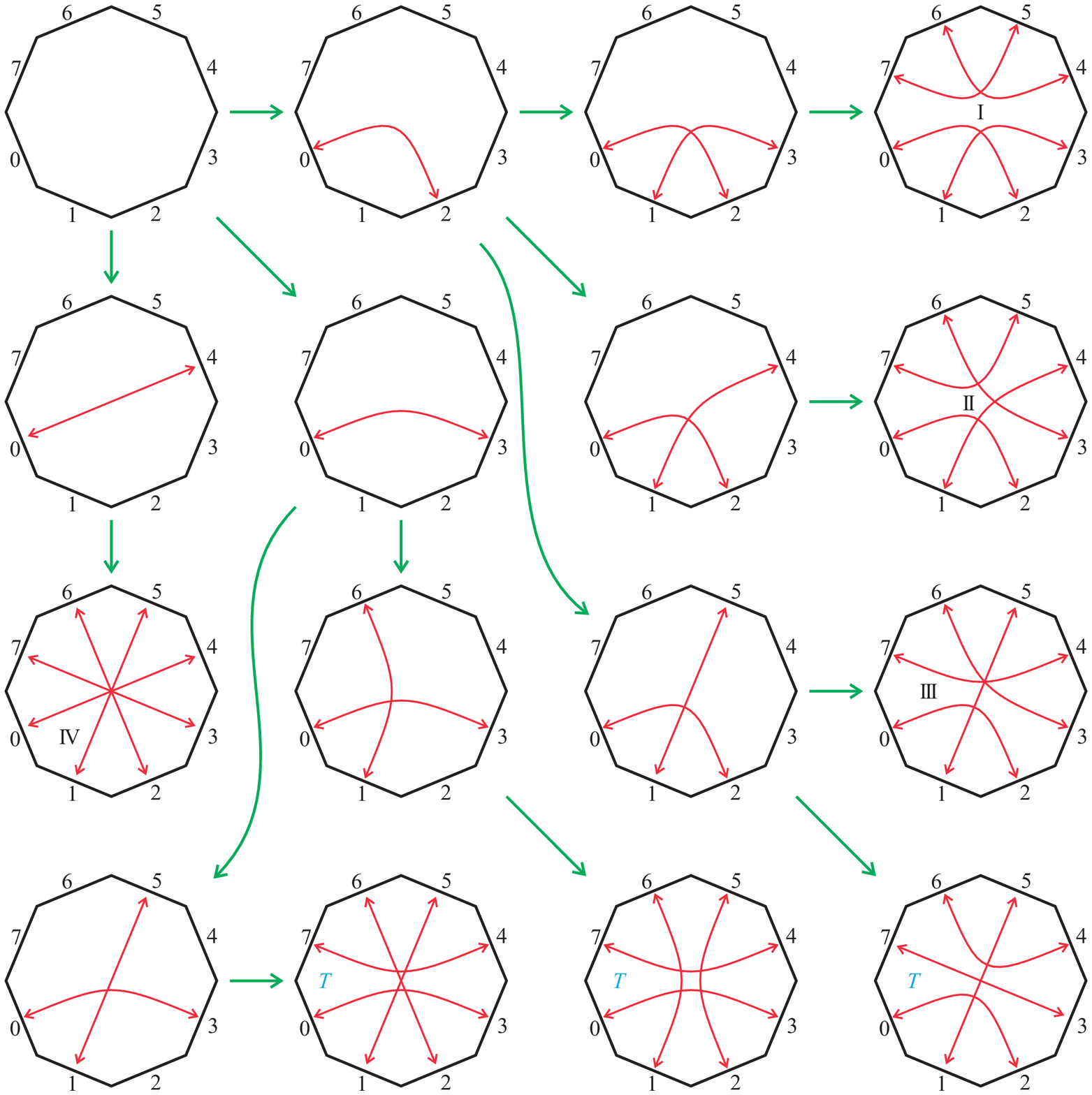}
    \end{center}
\mycap{Analysis of the edge-pairings of an octagon. \label{octaproof:fig}}
\end{figure}

Suppose that $i_1=2$. Up to symmetry we then have $3\leqslant i_3\leqslant 5$.
If $i_3=3$ we readily conclude that the other glued pairs are $\{4,6\},\{5,7\}$, so
we are in case I of Fig.~\ref{4octagons:fig}. If $i_3=4$ the other glued pairs
must be $\{3,6\},\{5,7\}$ and we are in case \secon. If $i_3=5$ the other glued pairs
can be $\{3,6\},\{4,7\}$, whence \terzo, or $\{3,7\},\{4,6\}$, but then there are 3 equivalence
classes of vertices, so this case must be dismissed.

Now suppose $i_1=3$. If $i_3=4$ then two edges in $\{5,6,7\}$ are paired, absurd.
If $i_3=5$ the other pairs can only be $\{2,6\},\{4,7\}$ but then there are 3 equivalence
classes of vertices, and the same happens with $i_3=6$ and other forced pairs $\{2,5\},\{4,7\}$.

For $i_1=4$ we of course get only \quart\ and the claim is proved.

\medskip

Note that \secon\ and \terzo\ have an order-2 symmetry, while
\primo\ has an order-4 dihedral symmetry and \quart\ has an order-16 dihedral symmetry.
This remark easily implies that there are $13$ inequivalent embeddings of the $(9,1)$ graph
in $2T$ with a disc as only region, obtained from \primo, \secon, \terzo, \quart\
by adding a small leg (with a label $e$ not shown)
in one of the positions described in Fig.~\ref{91emb:fig}.
\begin{figure}
    \begin{center}
    \includegraphics[scale=.6]{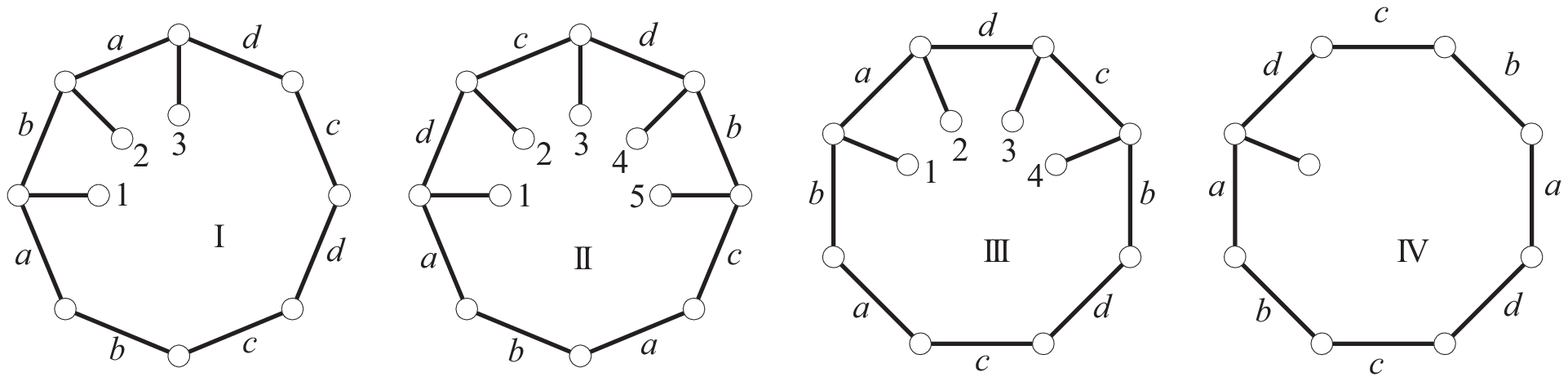}
    \end{center}
\mycap{Inequivalent embeddings in $2T$ of the $(9,1)$ graph with a disc as a only region. \label{91emb:fig}}
\end{figure}
Moreover these embeddings are subject to the following symmetries (and only them):
$$\primo.1(a,b,c,d,e)\leftrightarrow\primo.1(b,a,d,c,e)\qquad
\primo.3(a,b,c,d,e)\leftrightarrow\primo.3(d,c,b,a,e)$$
$$\secon.1(a,b,c,d,e)\leftrightarrow\secon.1(d,c,b,a,e)\qquad
\secon.5(a,b,c,d,e)\leftrightarrow\secon.5(d,c,b,a,e)$$
$$\quart(a,b,c,d,e)\leftrightarrow\quart(d,c,b,a,e).$$
Each of the $8$ cases without symmetries contributes to $\nu$ with the number of ordered
$5$-tuples with sum $k$, whence an $8\binom{k-1}4$ summand. To compute the contribution
of the cases with symmetries, which is the same for all of them,
we use the notation of $\primo.1$, so we must count the $5$-tuples $(a,b,c,d,e)$
with sum $k$ up to the symmetry $(a,b,c,d,e)\leftrightarrow(b,a,d,c,e)$.
If $b\neq a$ we can take $a>b$, while for $b=a$ we can take $c\leqslant d$, whence
$$\sum_{e=1}^{k-4}\left(
\sum_{a=2}^{k-e-3}\sum_{b=1}^{\min\{a-1,k-e-a-2\}}(k-e-a-b-1)+\sum_{a=1}^{[(k-e-2)/2]}\left[\frac{k-e-2a}2\right]\right).$$
We concentrate on the first sum and distinguish according to the parity of $k$.
For $k=2t$ we split the sum on $e$ between the even and the odd values of $e$,
so for $e=2j$ and $e=2j-1$. Imposing $2j\leqslant 2t-4$ and $2j-1\leqslant 2t-4$ we get
$j\leqslant t-2$ in both cases, while the inequality
$a-1\leqslant k-e-a-2$ is equivalent respectively to
$a\leqslant t-j-1$ and $a\leqslant t-j$. So we get
\begin{eqnarray*}
  \sum_{j=1}^{t-2} &  & \Bigg( \sum_{a=2}^{t-j-1}\sum_{b=1}^{a-1}(2t-2j-a-b-1)  \\
    & & +\sum_{a=t-j}^{2t-2j-3}\sum_{b=1}^{2t-2j-a-2}(2t-2j-a-b-1)\\
    & & +\sum_{a=2}^{t-j}\sum_{b=1}^{a-1}(2t-2j-a-b) \\
    & & +\sum_{a=t-j+1}^{2t-2j-2}\sum_{b=1}^{2t-2j-a-1}(2t-2j-a-b)\Bigg).
\end{eqnarray*}
One can now show that this repeated sum actually equals the expression
$\frac16(t-1)(t-2)(2t^2-6t+3)$, which can be done in two ways:
\begin{itemize}
\item Since $\sum_{s=1}^ms^p$ is a polynomial of degree $p+1$ in $m$,
the given repeated sum is a polynomial of degree $4$ in $t$; moreover it
vanishes at $t=1$ and $t=2$, so it is enough to make sure that the values
it attains at $t=3$, $t=4$ and $t=5$ coincide with those of
$\frac16(t-1)(t-2)(2t^2-6t+3)$;
\item One can substitute in the given repeated sum
the explicit expression of $\sum_{s=1}^ms^p$
as a polynomial of degree $p+1$ in $m$, and carry on the computation, which can also
be achieved in an automated way.
\end{itemize}
For $k=2t+1$ we similarly have
\begin{eqnarray*}
  \sum_{j=1}^{t-2} &  & \Bigg( \sum_{a=2}^{t-j}\sum_{b=1}^{a-1}(2t-2j-a-b)  \\
    & & +\sum_{a=t-j+1}^{2t-2j-2}\sum_{b=1}^{2t-2j-a-1}(2t-2j-a-b)\Bigg)\\
+\sum_{j=1}^{t-1}    & & \Bigg(\sum_{a=2}^{t-j}\sum_{b=1}^{a-1}(2t-2j-a-b+1) \\
    & & +\sum_{a=t-j+1}^{2t-2j-1}\sum_{b=1}^{2t-2j-a}(2t-2j-a-b+1)\Bigg)
\end{eqnarray*}
which is computed to be $\frac13t(t-1)^2(t-2)$, in either of the two ways already indicated above.
Replacing $t=\frac k2$ in the first formula and $t=\frac{k-1}2$ in the second one gives
$\frac1{48}\left(k^4 - 12k^3 + 50k^2 - 84k + x\right)$ with $x=48$ for even $k$ and $x=45$ for odd $k$.
For the second sum we can proceed likewise, getting $\frac13t(t-1)(t-2)$ for $k=2t$,
and $\frac16t(t-1)(2t-1)$ for $k=2t+1$, whence $\frac1{24}\left(k^3 - 6k^2 + yk+z\right)$
with $y=8,\ z=0$ for even $k$ and $y=11,\ z=-6$ for odd $k$. Putting the contributions
together it is now a routine matter to see that the final formula can be written as stated.

\noindent
Dipartimento di Matematica\\
Universit\`a di Pisa\\
Largo Bruno Pontecorvo, 5\\
56127 PISA -- Italy\\
\texttt{petronio@dm.unipi.it}

\end{document}